\documentclass[11pt]{article}
\usepackage[margin=1in]{geometry} 
\usepackage{amssymb,amsfonts,amsmath,bbm,mathrsfs,stmaryrd}
\usepackage{xcolor}
\usepackage{url}

\usepackage{enumerate}

\usepackage[colorlinks,
             linkcolor=blue,
             citecolor=black!75!red,
             pdfproducer={pdfLaTeX},
             pdfpagemode=None,
             bookmarksopen=true
             bookmarksnumbered=true]{hyperref}

\newcommand\bN{\mathbb N}

\newcommand\bZ{\mathbb Z}

\usepackage{tikz}
\usetikzlibrary{arrows,calc,decorations.pathreplacing,decorations.markings,intersections,shapes.geometric,through,fit,shapes.symbols,positioning,decorations.pathmorphing}

\usepackage[amsmath,thmmarks,hyperref]{ntheorem}
\usepackage{cleveref}

\creflabelformat{enumi}{#2(#1)#3}

\crefname{section}{Section}{Sections}
\crefformat{section}{#2Section~#1#3} 
\Crefformat{section}{#2Section~#1#3} 

\crefname{subsection}{\S}{\S\S}
\AtBeginDocument{%
  \crefformat{subsection}{#2\S#1#3}%
  \Crefformat{subsection}{#2\S#1#3}%
}

\theoremstyle{plain}

\newtheorem{lemma}{Lemma}[section]
\newtheorem{theorem}[lemma]{Theorem}
\newtheorem{proposition}[lemma]{Proposition}

\makeatletter
\let\xx@thm\@thm
\AtBeginDocument{\let\@thm\xx@thm}
\makeatother

\theoremstyle{nonumberplain}
\newtheorem{theoremN}{Theorem}

\newtheorem{qf'bis}{\Cref{le.quot_filt'} bis}

\theoremstyle{plain}
\theorembodyfont{\upshape}
\theoremsymbol{\ensuremath{\blacklozenge}}

\newtheorem{definition}[lemma]{Definition}

\newtheorem{example}[lemma]{Example}

\crefname{definition}{definition}{definitions}
\crefformat{definition}{#2definition~#1#3} 
\Crefformat{definition}{#2Definition~#1#3} 

\crefname{ex}{example}{examples}
\crefformat{example}{#2example~#1#3} 
\Crefformat{example}{#2Example~#1#3} 

\crefname{remark}{remark}{remarks}
\crefformat{remark}{#2remark~#1#3} 
\Crefformat{remark}{#2Remark~#1#3} 

\crefname{convention}{convention}{conventions}
\crefformat{convention}{#2convention~#1#3} 
\Crefformat{convention}{#2Convention~#1#3}

\crefname{lemma}{lemma}{lemmas}
\crefformat{lemma}{#2lemma~#1#3} 
\Crefformat{lemma}{#2Lemma~#1#3} 

\crefname{proposition}{proposition}{propositions}
\crefformat{proposition}{#2proposition~#1#3} 
\Crefformat{proposition}{#2Proposition~#1#3} 

\crefname{corollary}{corollary}{corollaries}
\crefformat{corollary}{#2corollary~#1#3} 
\Crefformat{corollary}{#2Corollary~#1#3} 

\crefname{theorem}{theorem}{theorems}
\crefformat{theorem}{#2theorem~#1#3} 
\Crefformat{theorem}{#2Theorem~#1#3} 

\crefname{enumi}{}{}
\crefformat{enumi}{(#2#1#3)}
\Crefformat{enumi}{(#2#1#3)}

\crefname{assumption}{assumption}{Assumptions}
\crefformat{assumption}{#2assumption~#1#3} 
\Crefformat{assumption}{#2Assumption~#1#3} 

\crefname{equation}{}{}
\crefformat{equation}{(#2#1#3)} 
\Crefformat{equation}{(#2#1#3)}

\theoremstyle{nonumberplain}
\theoremsymbol{\ensuremath{\blacksquare}}

\newtheorem{proof}{Proof}
\newcommand\pf[1]{\newtheorem{#1}{Proof of \Cref{#1}}}

\numberwithin{equation}{section}

\title{Free limits of free algebras}
\author{Alexandru Chirvasitu and Tao Hong}

\begin{document}

\date{}

\newcommand{\Addresses}{{
  \bigskip
  \footnotesize

  \textsc{Department of Mathematics, University at Buffalo, Buffalo,
    NY 14260-2900, USA}\par\nopagebreak \textit{E-mail address}:
  \texttt{achirvas@buffalo.edu}

  \medskip

  \textsc{Department of Mathematics, University at Buffalo, Buffalo,
    NY 14260-2900, USA}\par\nopagebreak \textit{E-mail address}:
  \texttt{thong5@buffalo.edu}

}}

\maketitle

\begin{abstract}
  Consider a diagram $\cdots \to F_3 \to F_2\to F_1$ of algebraic systems, where $F_n$ denotes the free object on $n$ generators and the connecting maps send the extra generator to some distinguished trivial element. We prove that (a) if the $F_i$ are free associative algebras over a fixed field then the limit in the category of graded algebras is again free on a set of homogeneous generators; (b) on the other hand, the limit in the category of associative (ungraded) algebras is a free formal power series algebra on a set of homogeneous elements, and (c) if the $F_i$ are free Lie algebras then the limit in the category of graded Lie algebras is again free.
\end{abstract}


\vspace{.5cm}



\section*{Introduction}
\label{se.intro}

This note was originally motivated by \cite[Theorem 1.1]{ess2} (or rather its earlier version, \cite[Theorem 1.2]{ess}):

\begin{theoremN}
  For any field $k$, the inverse limit in the category of graded commutative rings of the diagram
  \begin{equation}\label{eq:invlim}
    \cdots\to k[x_1,x_2,x_3]\to k[x_1,x_2]\to k[x_1]
  \end{equation}
  of polynomial rings (obtained by annihilating the extra variable at each step) is again a polynomial ring. 
\end{theoremN}

This is a curious and rather unexpected phenomenon: polynomial rings are the free objects in the category of commutative algebras (`free' in the sense of universal algebra, e.g. \cite[Definition 7.8.3]{berg-inv}), and hence expressible as coproducts. On the other hand, \Cref{eq:invlim} is a {\it limit} (rather than a {\it co}limit), so one would not necessarily expect compatibility between the two.

The natural question arises of which other varieties of algebras (apart from graded commutative algebras) exhibit the same type of freeness behavior: given free objects $F_n$ on sets of $n$ elements respectively, one can construct analogous diagrams
\begin{equation*}
  \cdots\to F_3\to F_2\to F_1
\end{equation*}
provided the algebras in question are equipped with a distinguished element $e$: the extra free generator  of $F_{n+1}$ can be sent to $e\in F_n$, giving the map $F_{n+1}\to F_n$. Examples include
\begin{itemize}
\item any number of ``linear'' varieties of algebras (commutative, associative, Jordan, etc.), with $0$ as the trivial element;
\item groups, monoids, and so on, with their respective trivial elements. 
\end{itemize}

Groups, in particular, have been studied from this perspective by Higman: consider a limit
\begin{equation*}
  \lim\left(\cdots\to G_1*G_2*G_3\to G_1*G_2\to G_1\right)
\end{equation*}
in the category of groups, where
\begin{itemize}
\item $G_i$ are groups;
\item `$*$' denotes the coproduct (or free product) of groups;
\item each connecting map annihilates the extra free factor and acts as the identity on the others. 
\end{itemize}
In the language of \cite{hig-unr}, that limit is the {\it unrestricted free product} of the groups $G_i$. In particular, when all $G_i$ are isomorphic to $\bZ$, one obtains a kind of completion of a free group on countably-infinitely many generators, denoted here by $F$. The analogue of the question posed above is whether $F$ is again free. Higman shows that not only is this not the case, but in fact $F$ is in a sense at the opposite end of a freeness spectrum (\cite[Theorem 1]{hig-unr}):

\begin{theoremN}
  Let $F_n$ be the free group on $n$ generators, and consider the limit
  \begin{equation*}
    F:=\lim\left(\cdots \to F_3\to F_2\to F_1\right)
  \end{equation*}
  in the category of groups. Then, any morphism from $F$ into a free group factors through one of the $F_n$.
\end{theoremN}

This shows that the freeness result recorded in \cite{ess,ess2} is far from being a given. Here, we prove a number of cognates. First, \Cref{th:gr-fr} is a direct non-commutative analogue of \cite[Theorem 1]{ess2}; the statement in the main text is more precise, but roughly, it reads:

\begin{theorem}
  Let $k$ be a field and $A_n$ the free algebra on $n$ generators. Then, the limit
  \begin{equation}\label{eq:assoclim}
    \lim\left(\cdots\to A_3\to A_2\to A_1\right)
  \end{equation}
  in the category of graded associative algebras is free on a set of homogeneous elements.
\end{theorem}

On the other hand, if one were to instead take the limit \Cref{eq:assoclim} in the category of {\it plain} (as opposed to graded) associative algebras, a trace of this freeness behavior survives (\Cref{th.ungr-fr}):

\begin{theorem}
  Let $k$ be a field and $A_n$ the free algebra on $n$ generators. Then, the limit
  \begin{equation}\label{eq:assoclim}
    \lim\left(\cdots\to A_3\to A_2\to A_1\right)
  \end{equation}
  in the category of associative algebras is a formal power series algebra on a set of homogeneous elements.
\end{theorem}

Finally, there is a graded-Lie-algebra version of the above (\Cref{th:lie-fr}):

\begin{theorem}
    Let $k$ be a field and $L_n$ the free Lie algebra on $n$ generators. Then, the limit
  \begin{equation}\label{eq:assoclim}
    \lim\left(\cdots\to L_3\to L_2\to L_1\right)
  \end{equation}
  in the category of graded Lie algebras is free on a set of homogeneous elements.
\end{theorem}

\subsection*{Acknowledgements}

This work was partially supported through NSF grant DMS-2001128.

The material constitutes part of the second author's PhD thesis at the University at Buffalo.

\section{Preliminaries}
\label{se.prel}

We work over an arbitrary field $k$ (all additional properties, e.g. being perfect, will be specified if and when needed). The algebras under consideration will often be graded, meaning here $\bN$-graded where $\bN=\{0,1,\cdots\}$. As \cite{chn-fr} will be a central reference throughout, we follow it in denoting by $\nu$ the ``highest-degree'' function on a filtered algebra
\begin{equation*}
  0=A_{(-\infty)}\subset A_{(0)}\subset A_{(1)}\subset \cdots, \ A=\cup_{n}A_{(n)}
\end{equation*}
i.e. $\nu(x)$ is the smallest $n$ such that $x\in A_{(n)}$. Our algebras will sometimes be graded, as in
\begin{equation*}
  A=\bigoplus_{d\in \bN}A_d,
\end{equation*}
in which case we consider the corresponding filtration given by
\begin{equation*}
  A_{(n)} = \bigoplus_{d=0}^n A_d. 
\end{equation*}
For homogeneous elements in the graded case we also sometimes resort to $|x|$ for the degree of $x$. Our graded algebras will also often be {\it connected}: $A_0=k$, the ground field.

We will need some auxiliary material from \cite[Chapter 2]{chn-fr}, which we now present briefly. First, recall the following discussion from \cite[$\S$2.2]{chn-fr}. 

\begin{definition}\label{def.dep}
  A family $(a_i)_i$ of elements of $A$ is {\it right $\nu$-dependent} if one of the $a_i$ vanishes, or there exist $b_i\in A$, almost all zero, such that
  \begin{equation*}
    \nu\left(\sum a_ib_i\right) < \max_i(\nu(a_i)+\nu(b_i)). 
  \end{equation*}

  An element $a\in A$ is {\it right $\nu$-dependent} on the family $(a_i)_i$ if $a=0$ or there are $b_i\in A$, almost all zero, such that
  \begin{equation*}
    \nu\left(a-\sum a_ib_i\right) < \nu(a),\quad \max_i(\nu(a_i)+\nu(b_i))\le \nu(a). 
  \end{equation*}
\end{definition}

The two properties above are ordered strength-wise: the right $\nu$-dependence of $a$ on $(a_i)$ entails the right $\nu$-dependence of the family $(a_i)\cup\{a\}$. Rings that in a certain sense satisfy the converse of this observation are the focus of \cite[Chapter 2]{chn-fr}, as they tend to have good ``freeness'' properties.

\begin{definition}\label{def.wk}
  The filtered algebra $A$ {\it has (or satisfies) the weak algorithm} for $\nu$ if, for any $\nu$-dependent family $(a_i)$ in the sense of \Cref{def.dep} some $a_i$ is $\nu$-dependent on the family of those $a_j$ with $\nu(a_j)\le \nu(a_i)$. 
\end{definition}

As hinted above, the importance of the concept for us is encapsulated by the following result (\cite[Proposition 2.4.2]{chn-fr}). 

\begin{proposition}\label{pr.chn}
  A filtered $k$-algebra $A$ with $A_{(0)}=k$ is free on some subset $X$ of positive-degree elements if and only if it satisfies the weak algorithm.
\end{proposition}

A parallel discussion can be carried out for algebras equipped with an {\it inverse filtration} as in \cite[$\S$2.9]{chn-fr}:
\begin{equation*}
  A=A_{(0)}\supset A_{(1)}\supset \cdots\supset 0=A_{(\infty)}. 
\end{equation*}
In this case for $x\in A$ we denote by $\nu(x)$ the largest $n$ such that $x\in A_{(n)}$. $\nu$-dependence can be defined as before by simply reversing the inequalities. We then have

\begin{definition}\label{def.inv-wk}
  The inversely filtered algebra $A$ {\it has the inverse weak algorithm} for $\nu$ if for any $\nu$-dependent family $(a_i)$ in the sense of \Cref{def.dep} with reversed inequalities some $a_i$ is $\nu$-dependent on the family of those $a_j$ with $\nu(a_j)\le \nu(a_i)$. 
\end{definition}

In the inverse-filtration setting the analogue of \Cref{pr.chn} reads as follows (see \cite[Proposition 2.9.8]{chn-fr}).

\begin{proposition}\label{pr.chn-inv}
  A complete inversely filtered $k$-algebra $A$ with $A/A_{(1)}\cong k$ is a formal power series algebra on some subset $X$ if and only if it satisfies the inverse weak algorithm. 
\end{proposition}

\section{Associative algebras}
\label{se.main}

\subsection{The graded case}
\label{subse.gr}

The setup is as follows. We fix a set $S$, and consider the free $k$-algebra $k\langle S\rangle$ on $S$. For an arbitrary finite subset $F\subset S$ consider the surjection $k\langle S\rangle\to k\langle F\rangle$ obtained by annihilating all generators in $S\setminus F$. These surjections form a co-filtered diagram in the category of graded algebras, and we can consider its limit
\begin{equation}\label{eq:diag}
  A=A((S)):=\varprojlim_F \left(k\langle S\rangle \to k\langle F\rangle\right). 
\end{equation}
The degree-$d$ component $A_d$ consists of formal $k$-linear combinations of the degree-$d$ monomials in the generators $x_s$, $s\in S$. 

There is a canonical morphism $k\langle S\rangle\to A((S))$ that is evidently an isomorphism when $S$ is finite and only one-to-one when $S$ is infinite. In the latter case however, the algebra $A((S))$ is still free. This is the content of the main result of this section, which is a non-commutative analogue of \cite[Theorem 1.1]{ess2} (and of its precursor, \cite[Theorem 1.2]{ess}).

\begin{theorem}\label{th:gr-fr}
  $A=A((S))$ defined by \Cref{eq:diag} is free as a $k$-algebra, on any set $X$ of homogeneous elements of $A$ forming a basis for $A_{>0}/A_{>0}^2$ is a free generating set for $A$.
\end{theorem}
\begin{proof}
  We use \cite[\S 2.4, Theorem 4.1]{chn-fr}. To apply it, we have to prove that
  \begin{enumerate}[(a)]
  \item\label{item:1} $A$ satisfies the weak algorithm with respect to its grading;
  \item\label{item:2} the monomials on any set $X$ as in the statement span $A$,
  \item\label{item:3} and no element of $X$ is right-$\nu$-dependent on the rest.
  \end{enumerate}
  We handle these in turn.

  {\bf \Cref{item:1}: $A$ satisfies the weak algorithm.} Suppose $a_i$, $1\le i\le n$ form a $\nu$-dependent family of non-zero homogeneous elements, ordered so that
  \begin{equation*}
    |a_1|\le \cdots\le |a_n|.
  \end{equation*}
  We will argue that one of the $a_i$ is $\nu$-dependent on $(a_j)_{j=1}^{i-1}$.

  By homogeneity, the hypothesis proves the existence of (homogeneous) $b_i$ such that $\sum a_ib_i=0$; we then need to show that some $a_i$ is a right linear combination of $a_j$, $j<i$. We may as well assume all $a_ib_i$ have the same degree, so that
  \begin{equation*}
    |b_1|\ge \cdots\ge |b_n|. 
  \end{equation*}
  We prove the statement by double induction on $n$ and then $|b_n|$, the base case being a simple exercise.

  If $b_n=0$ then we may as well restrict attention to the family $(a_i)_{i=1}^{n-1}$, the inductive hypothesis taking care of the rest. We can thus assume that $b_n\ne 0$, moving over to the induction-by-$|b_n|$ branch of the argument. Once more, the base case $b_n\in k^\times$ is immediate (as there is then nothing to prove), so we in fact assume $|b_n|>0$. 

  Fix some $s$ such that $x_s$ appears as the rightmost generator in one of the monomials making up $b_n\ne 0$. In the relation $\sum a_ib_i=0$ the terms ending in $x_s$ still add up to zero, so we can ignore the summands of the $b_i$s that do {\it not} end in a rightmost $x_s$ and assume all $b_i$ belong to $Ax_s$. But then, if say $b_i=b_i'x_s$, we have
  \begin{equation*}
    \sum a_ib_i'=0,
  \end{equation*}
  allowing us to apply the inductive hypothesis to $|b'_n|<|b_n|$.

  \vspace{.5cm}

{\bf \Cref{item:2}: The monomials on $X$ span $A$.} This is virtually automatic given the definition of $X$. Let $a\in A_d$ be a homogeneous element. It is then of the form
\begin{equation*}
  \sum t_ix_i + a',\ t_i\in k,\ x_i\in X
\end{equation*}
for some $a'\in A_{>0}^2$, and we can use induction on $d$ applied to the factors $b_j,c_j\in A_{>0}$ in a decomposition
\begin{equation*}
  a'=\sum_j b_jc_j. 
\end{equation*}
    
\vspace{.5cm}

{\bf \Cref{item:3}: No element of $X$ is right $\nu$-dependent on the rest.} because our elements are homogeneous, a $\nu$-dependence relation would be of the form 
\begin{equation*}
  x=\sum x_ia_i,\ a_i\in A
\end{equation*}
for distinct elements $x,x_i\in X$. Modulo $A_{>0}^2$ this expresses $x$ as a $k$-linear combination of other elements of $X$, contradicting the fact that $X$ is a basis of $A_{>0}/A_{>0}^2$.
\end{proof}

\subsection{The ungraded case}
\label{subse.ungr}

The preceding discussion makes it natural to examine the structure of the limit \Cref{eq:diag} in the category of (plain, ungraded) $k$-algebras. We will denote the resulting algebra by $B=B(S)$, to distinguish it from the limit $A=A((S))$ in the category of graded algebras.

We have an inclusion $A((S))\subset B(S)$. In fact, the generic element of $B$ is a formal series of the form
\begin{equation*}
  x=\sum_{d\ge 0}x_d,\ x_d\in A_d
\end{equation*}
with the property that for every finite subset $F\subset S$, the free algebra $k\langle F\rangle\subset A\subset B$ contains only finitely many of the monomials appearing in $x$.

$B$ admits a natural inverse filtration $\nu$ defined by
\begin{equation*}
  x=\sum_{d\ge 0}x_d\in B_{(n)} \iff x_d=0\text{ for }d<n. 
\end{equation*}
In other words,
\begin{equation*}
  \nu(x) = \text{ smallest }d\text{ such that }x_d\ne 0. 
\end{equation*}
We denote
\begin{equation*}
  K=B_{(0)},\ B_+ = \{x\in B\ |\ \nu(x)>0\}. 
\end{equation*}
Note that $K$ is a division ring. 

We then have the following analogue of \Cref{th:gr-fr}. 

\begin{theorem}\label{th.ungr-fr}
  The algebra $B=B(S)$ is isomorphic to a formal power series $k$-algebra. 
\end{theorem}
\begin{proof}
  According to \Cref{pr.chn-inv} it suffices to argue that $B$ has the weak algorithm for its inverse filtration $\nu$ (as it is clear that $B$ is $\nu$-complete). Suppose , then, that we have
  \begin{equation*}
    \nu\left(\sum a_ib_i\right) > \min_i(\nu(a_i)+\nu(b_i))
  \end{equation*}
  for non-zero $a_i$ and $b_i$, $1\le i\le n$. We may as well assume that all $\min_i(\nu(a_i)+\nu(b_i))$ are equal and that $\nu(a_i)$ is non-decreasing in $i$.

  If $\nu(b_n)=0$ then $b_n$ is invertible, and hence $a_n$ is a right $B$-linear combination of the other $a_i$ modulo $B_{(\nu(a_n)+1)}$, proving $\nu$-dependence.
  
  Otherwise, the lowest-degree term of $b_n$ is a $k$-linear combination of monomials in the $x_s$, $s\in S$ of degree $\ge 1$. Let $x_t$ be the rightmost generator in one of these monomials. Denoting by a $t$ superscript the linear combination of those monomials that end in $x_t$, we then have
  \begin{equation*}
    \nu\left(\sum a_ib_i^t\right) > \min_i(\nu(a_i)+\nu(b_i^t)). 
  \end{equation*}
  Since all $b_i^t$ are of the form $b'_i x_t$, we can eliminate $x_t$ to obtain
  \begin{equation*}
    \nu\left(\sum a_ib'_i\right) > \min_i(\nu(a_i)+\nu(b'_i)). 
  \end{equation*}
  This is a $\nu$-dependence relation as before with $\nu(b'_n)<\nu(b_n)$, so we can repeat the procedure until the dependence relation has been reduced to the case $\nu(b_n)=0$.
\end{proof}

\section{Lie algebras}\label{se:lie}

The goal here is to prove an Lie-algebra analogue of \Cref{th:gr-fr}. Specifically, let $S$ be a set as before. If $S$ is finite then $L((S))$ will simply denote the free Lie algebra on $S$. On the other hand, if $S$ is infinite we set

\begin{equation}
  \label{eq:1}
  L=L((S)):=\varprojlim_F \left(L((S)) \to L(F)\right),
\end{equation}
precisely as in \Cref{eq:diag}: 

\begin{itemize}
\item the limit is in the category of {\it graded} Lie algebras;  
\item it is indexed by the filtered system of finite subsets $F\subseteq S$;
\item the connecting morphisms $L(F)\to L(F')$ for $F'\subset F$ simply annihilate all generators in $F\setminus F'$. 
\end{itemize}

Furthermore, we assume we are working over a base field $k$ of characteristic zero. The advertised result is

\begin{theorem}\label{th:lie-fr}
  Let $S$ be a set and $L=L((S))$ as in \Cref{eq:1}. Then, $L$ is freely generated as a Lie algebra by any set of homogeneous elements that projects to a basis of $L/[L,L]$.
\end{theorem}

We will make use of the material in \cite{reut}, which is a good reference for free Lie algebras in general. Following that source, we reserve the notation $L(S)$ for the free Lie algebra on the (finite or infinite) set $S$. In particular, for finite $S$ we have $L((S))=L(S)$. As explained in \cite[\S 1.2]{reut}, the free Lie algebra $L(S)$ can be identified with the set of {\it Lie polynomials} in its enveloping algebra $k\langle S\rangle$. 

The proof given there for the celebrated theorem of Shirshov and Witt to the effect that Lie subalgebras of free Lie algebras are free \cite[Theorem 2.5]{reut} proceeds via a Lie-theoretic version of Cohen's theory of dependence for polynomials. The relevant appears in \cite[discussion preceding Theorem 2.3]{reut}:

\begin{definition}\label{def:lie-dep}
  Let $p,p_i\in k\langle S\rangle$, $1\le i\le n$ be non-commutative polynomials. We say that $p$ is {\it Lie-dependent} on the $p_i$ if $p=0$ or there is a Lie polynomial $f$ in $n$ variables such that
  \begin{itemize}
  \item $\deg(p-f(p_1,\cdots,p_n))<\deg p$;
  \item the degree of every monomial appearing formally in $f(p_i)$ is dominated by the degree of $p$.      
  \end{itemize}
\end{definition}

The main tool in the proof of the above-mentioned \cite[Theorem 2.5]{reut} is the following Lie analogue of Cohen's result \Cref{pr.chn} (see \cite[Theorem 2.3]{reut}). 

\begin{proposition}\label{pr:lie-dep}
  Let $p_i\in k\langle S\rangle$, $1\le i\le n$ be Lie polynomials. If the family is dependent then some $p_i$ is Lie-dependent on those $p_j$ of no-larger degree.
\end{proposition}

In order to state our first result, note that we have a natural embedding
\begin{equation}\label{eq:2}
  L((S))\subset A((S))
\end{equation}
arising by taking the limit of the Lie-polynomial embeddings $L(F)\subset k\langle F\rangle$ over finite subsets $F\subset S$. We say that $L((S))$ consists of the {\it Lie elements} of $A((S))$. 

Our first goal will be to prove the following version of \Cref{pr:lie-dep}.

\begin{theorem}\label{th:lie-is-dep}
Let $S$ be a set and $p_i\in A((S))$, $1\le i\le n$ be Lie elements. If the family is dependent then some $p_i$ is Lie-dependent on those $p_j$ of no-larger degree.  
\end{theorem}
\begin{proof}
  As in the proof of \cite[Theorem 2.3]{reut}, it will be enough to argue that, upon denoting the top homogeneous component of an element $q\in A((S))$ by $\overline{q}$, some $\overline{p_i}$ is expressible as a Lie polynomial of $\overline{p_j}$ for $j\ne i$.

  To simplify the notation we assume from the start that the $p_i$ are homogeneous, as we well may: the hypothesis is that there are $b_i\in A((S))$, almost all zero, such that
  \begin{equation}\label{eq:smldeg}
    \deg \sum_i p_i b_i < \max_i \deg p_i b_i. 
  \end{equation}
  Retaining only the top-degree components of the $p_i$ and $b_i$ and only those indices $j$ such that
  \begin{equation*}
    \deg p_j b_j = \max_i \deg p_i b_i,
  \end{equation*}
  we obtain
  \begin{equation*}
    \sum_i \overline{p_i} b_i=0,
  \end{equation*}
  i.e. the top components $\overline{p_i}$ are dependent. We henceforth drop the overlines and assume all $p_i$ are homogeneous, satisfying 
  \begin{equation*}
    \sum_i p_i b_i=0
  \end{equation*}
  for some homogeneous $b_i$, almost but not all zero. The goal (sufficient, again as in the proof of \cite[Theorem 2.3]{reut}) will be to prove that some $p_i$ is a Lie polynomial in $p_j$, $j\ne i$. This follows from \Cref{pr:idlie} below.
\end{proof}

The following result is a version of \cite[Lemma 2.4]{reut}.

\begin{lemma}\label{pr:idlie}
  Let $S$ be a set and $p$, $p_i\in A((S))$, $1\le i\le n$ be homogeneous Lie elements with
  \begin{equation}\label{eq:lindep}
    p = \sum_i p_i b_i
  \end{equation}
  for homogeneous $b_i\in A((S))$. Then,
  \begin{equation}\label{eq:pisliepoly}
    p = f(p_1,\cdots,p_n)
  \end{equation}
  for some Lie polynomial $f$.
\end{lemma}
\begin{proof}
  For finite $S$ the argument appears in the course of the proof of \cite[Theorem 2.3]{reut}, so we are concerned with infinite $S$. For finite subsets $F\subset S$ we denote the images through the surjection $A((S))\to k\langle F\rangle$ with superscripts $F$, as in $p^F$, $p_i^F$, etc.

  We assume throughout the discussion that the finite subsets $F$ are large enough to ensure that the projected elements $p^F$, $p_i^F$ and $b_i^F$ have the same degrees as their global counterparts $p$, $p_i$, etc. The relation \Cref{eq:lindep} projects to analogues
  \begin{equation*}
    p^F = \sum_i p^F_i b^F_i\in k\langle F\rangle, 
  \end{equation*}
  and hence, according to the aforementioned proof of \cite[Theorem 2.3]{reut}, for each $F$ we have an expression
  \begin{equation*}
    p^F = f_F(p_1^F,\cdots,p_n^F)
  \end{equation*}
  for some Lie polynomial $f_F$. We will furthermore assume $f_F$ is chosen {\it minimally}, in the sense that its Lie monomials, evaluated at the $p_i^F$, produce linearly independent elements of $k\langle F\rangle$. 

  If $d$ and $d_i$ denote the degrees of the homogeneous elements $p$ and $p_i$ respectively, then every Lie monomial appearing in $f_F$ has some degree $e_i$ in $x_i$ so that
  \begin{equation*}
    d = \sum_i d_i e_i. 
  \end{equation*}
  There are only finitely many choices of such monomials, so only finitely many monomials appearing in all $f_F$ collectively. This means that there is some {\it cofinal} collection $F_{\alpha}$ of $F$ (i.e. such that every finite subset $F\subset S$ is contained in some $F_{\alpha}$) for which
  \begin{itemize}
  \item the Lie monomials appearing in $f_{F_{\alpha}}$ are the same for all $\alpha$;
  \item by minimality, the coefficients of those monomials are the same for all $\alpha$. 
  \end{itemize}
  In other words, all Lie polynomials $f_{F_{\alpha}}$ coincide with some common Lie polynomial $f$. But this means that
  \begin{equation*}
    p^{F_{\alpha}} = f(p_1^{F_{\alpha}},\cdots,p_n^{F_{\alpha}}),\ \forall \alpha,
  \end{equation*}
  and hence we obtain the desired identity \Cref{eq:pisliepoly}.
\end{proof}
We first prove a weak version of \Cref{th:lie-fr}.
\begin{theorem}
  \label{th:weak}
  Let $S$ be a set and $L=L((S))$ as in \Cref{eq:1}. Then, $L$ is freely generated as a Lie algebra by some set of homogeneous elements.
\end{theorem}
\begin{proof}
	$E_n$ is a subspace of $L$ defined by 
	\[
	E_n=\{p\in L\subset A((S))|\deg(p)\leq n\}
	\]
	Let $\langle E\rangle$ denote the Lie subalgebra generated by $E\subset L$.
	Let $E'_n$ be the subspace of $E_n$ defined by 
	\begin{displaymath}
  E_n':=E_n \cap \langle E_{n-1}\rangle
     \end{displaymath}
     Let $X_n$ be a subset of $E_n$ which defines a basis of $E_n/E'_n$.
Let $X:=\cup_{n\geq 1}X_n$. In order to show $L$ is free on $X$, it is enough to show $L$ is isomorphic with $L(B)$, where $B$ is a set with bijection $b\mapsto x_b$, $B\to X$. We only need to show (i) $X$ generates $L$. (ii) for each nonzero Lie polynomial $f(b)_{b\in B}\in L(B)$, one has $f(x_b)_{b\in B}\neq 0$\\
For (i)\\
Let $p$ with $\deg(p)=n$. We will prove for each $n$, $E_n$ is generated by $X$. When $n=0$, clearly it is true. Suppose when $n=k-1$, $E_{k-1}$ generated by $X$. Since $X_k$ is a basis of $E_k/E'_{k-1}$, so for each $p\in E_k$, we have $\overline{p}=\sum \alpha_x \overline{x}$, here $\overline{p}$ is projection of $p$ in $E_k/E'_k$, so we have $$Q=P-\sum_{x\in X_k} \alpha_x x \in E'_k\subset \langle E_{k-1}\rangle.$$
Hence, $Q$ is generated by $E_{k-1}$. By induction $E_{k-1}$ is generated by $X$. This shows $E_k$ is generated by $X$, hence $X$ generates $L$.\\
For (ii)\\
Arguing by contradiction. Suppose $f(p_1\dots p_q)=0$, for some nonzero Lie polynomial $f(b_1\dots b_q)\in L(B)$ and some $p_1\dots p_q \in X$ with $\deg(p_1)\leq \cdots \leq \deg(p_q)$. Certainly, there exists a nonzero polynomial in $K(B)$ such that $f(p_1\dots p_q)=0$, take such polynomial with the least degree, we write $f$ as 
$$f=\sum_{i=1}^q b_ig_i .$$
Some $r_i=g_i(p_1\dots p_q)$ is nonzero. Otherwise, suppose $g_i(p_1\dots p_q)=0$ for each $i$, at least one $g_i$ is nonzero polynomial so there is a polynomial $g_i$ satisfying $g_i(p_1\dots p_q)=0$, but its degree is less than $\deg(f)$, it contradict the minimality of $f$. Since $0=f(p_1\dots p_q)=\sum p_i r_i$, we deduce $p_1\dots p_q$ are dependent. By \Cref{th:lie-is-dep}, some polynomial $p_i$ is Lie dependent on $p_1\dots p_{i-1}$. Hence,$\deg(p_i-h(p_1\dots p_{i_1})<\deg(p_i)=n$, here $p_i-h(p_i \dots p_{i-1})\in E_{n-1}$, $h$ is a Lie polynomial. Hence, $p_i-h(p_1\dots p_{i-1})=\text{an element of } E_{n-1}$. This implies $p_i + \text{a linear combinations  of those } p_j, j<i \text{ of the same degree as } p_i=\text{a Lie expression which are of degree less than } p_i + \text{an element of } E_{n-1}$ with $n=\deg(p_i)$. So we have polynomials in $X_n$ are not linearly independent in $E_n/E'_n$, which is a contradiction. Therefore, this proposition is proved.
\end{proof}

\pf{th:lie-fr}
\begin{th:lie-fr}  
  Since \Cref{th:weak} shows that we can choose a set $X$ of {\it homogeneous} free generators for $L((S))$, the conclusion follows from \Cref{pr:allhomog} below.
\end{th:lie-fr}

\begin{proposition}\label{pr:allhomog}
  Let $L$ be a $\bZ_{>0}$-graded Lie algebra freely generated by some set of homogeneous elements. Then, $L$ is freely generated by {\it any} set of homogeneous elements projecting to a basis of $L/[L,L]$.
\end{proposition}
\begin{proof}
  We abbreviate the phrase
  \begin{equation*}
    X \text{ projects to a basis of }L_{ab}:=L/[L,L].
  \end{equation*}
  to
  \begin{equation*}
    X \text{ is {\it relatively free}}.
  \end{equation*}
  Denote
  \begin{itemize}
  \item by $F$ a set of homogeneous elements generating $L$ freely (note that in particular $F$ is relatively free in the above sense);
  \item by $X$ an arbitrary relatively free set of homogeneous elements;
  \item by subscript degree decorations the respective homogeneous components of $F$ and $X$; for instance:
    \begin{equation*}
      X_d:=\{x\in X\ |\ |x| = d\};
    \end{equation*}
  \item by
    \begin{equation*}
      0<d_1<d_2<\cdots
    \end{equation*}
    the positive integers appearing as degrees of elements in $F\cup X$;
  \item by overlines (e.g. $\overline{X}$) the images of sets through the projection $L\to L_{ab}$. 
  \end{itemize}
  We will argue that there is an automorphism $\alpha$ of $L$ (as a graded Lie algebra) transforming $F$ into $X$; naturally, this will imply the desired conclusion.

  We define $\alpha$ on $F_{d_n}$, inductively on $n$. First, note that because $\overline{X_{d_1}}$ and $\overline{F_{d_1}}$ are bases for the same subspace of $L_{ab}$, there is a linear automorphism of $\mathrm{span}(F_{d_1})$ that transforms $F_{d_1}$ into $X_{d_1}$ modulo $[L,L]$. Furthermore, because $[L,L]$ is spanned by commutators of elements of $F$ and such commutators are all of degree
  \begin{equation*}
    \ge \min(d_2,2d_1) > d_1,
  \end{equation*}
  said automorphism must in fact transform $F_{d_1}$ into $X_{d_1}$.

  This would be the degree-$d_1$ component of $\alpha$; to simplify matters, we assume (as we can, by the argument in the preceding paragraph) that in fact $F_{d_1}=X_{d_1}$. This constitutes the base case in the recursive procedure we are outlining.

  Next, consider degree $d_2$. Each element of $X_{d_2}$ lies in
  \begin{equation*}
    \mathrm{span}(F_{d_2}) + \mathrm{span}([F_{d_1},F_{d_1}])
  \end{equation*}
  and vice versa:
  \begin{equation*}
    F_{d_2}\subset \mathrm{span}(X_{d_2}) + \mathrm{span}([X_{d_1},X_{d_1}]),
  \end{equation*}
  because
  \begin{itemize}
  \item $F_{d_1}=X_{d_1}$, and    
  \item $\overline{X_{d_2}}$ and $\overline{F_{d_2}}$ constitute bases for the same subspace
    \begin{equation}\label{eq:xfd2}
      \mathrm{span}\left(\overline{X_{d_2}}\right) = \mathrm{span}\left(\overline{F_{d_2}}\right)
    \end{equation}
    of $L_{ab}$.
  \end{itemize}
  An automorphism of \Cref{eq:xfd2} that identifies $\overline{F_{d_2}}$ with $\overline{X_{d_2}}$ can thus be lifted to an automorphism of
  \begin{equation*}
    \mathrm{span}(F_{d_1})+\mathrm{span}(F_{d_2}) + \mathrm{span}([F_{d_1},F_{d_1}])
  \end{equation*}
  that
  \begin{itemize}
  \item is the identity on $F_{d_1}=X_{d_1}$,
  \item respects the Lie bracket on that space, 
  \item and maps $F_{d_2}$ to $X_{d_2}$.
  \end{itemize}
  This procedure can similarly be continued recursively to higher degrees: we henceforth assume that $X_{d_i}=F_{d_i}$ for $i=1,2$, etc.
\end{proof}

We observe that homogeneity is crucial to \Cref{pr:allhomog}:

\begin{example}
  Consider the free Lie algebra $L$ on two generators, $x$ and $y$. The elements
  \begin{equation*}
    x':=x+[x,[x,y]]\quad \text{and}\quad y':=y+[x,[x,[x,y]]]
  \end{equation*}
  surject to a basis for $L_{ab}$, but any non-trivial Lie polynomial in $x'$ and $y'$ will have degree at least 3. It follows that $[x,y]$, for instance, cannot lie in the (free) Lie algebra generated by $x'$ and $y'$.
\end{example}



\def\polhk#1{\setbox0=\hbox{#1}{\ooalign{\hidewidth
  \lower1.5ex\hbox{`}\hidewidth\crcr\unhbox0}}}

\Addresses
  
\end{document}